\documentclass{article}
\usepackage[english]{babel}
\usepackage{amsmath}
\usepackage{amssymb}
\usepackage{amsfonts}
\usepackage{amsthm}
\usepackage{graphicx}

\usepackage{url}

\def\thtext#1{
  \catcode`@=11
  \gdef\@thmcountersep{. #1}
  \catcode`@=12
}

\def\threst{
  \catcode`@=11
  \gdef\@thmcountersep{.}
  \catcode`@=12
}

\theoremstyle{plain}
\newtheorem{thm}{Theorem}[section]
\newtheorem{prop}[thm]{Proposition}
\newtheorem{cor}[thm]{Corollary}

\theoremstyle{definition}

\newtheorem{dfn}[thm]{Definition}
\newtheorem{rk}[thm]{Remark}

 \catcode`@=11
 \def\.{.\spacefactor\@m}
 \catcode`@=12

\def\R{\mathbb R}

\def\a{\alpha}

\def\e{\varepsilon}

\def\D{\Delta}
\def\g{\gamma}

\def\l{\lambda}

\def\s{\sigma}

\def\0{\emptyset}
\def\:{\colon}
\def\<{\langle}
\def\>{\rangle}

\def\rom#1{\emph{#1}}
\def\({\rom(}
\def\){\rom)}

\def\ss{\subset}

\def\x{\times}

\def\bcAD{\overline{\mathstrut\cAD}}

\def\bF{{\bar F}}

\def\diam{\operatorname{diam}}

\def\Ext{\operatorname{Ext}}

\def\MST{\operatorname{MST}}
\def\mst{\operatorname{mst}}

\def\cAD{\mathcal{AD}}

\def\cD{\mathcal{D}}

\def\cM{\mathcal{M}}

\begin{document}
\title{The Gromov--Hausdorff Distances between Simplexes and Ultrametric Spaces}
\author{Alexander O.~Ivanov, Alexey A.~Tuzhilin}
\date{}
\maketitle

\begin{abstract}
In the present paper we investigate the Gromov--Hausdorff distances between a bounded metric space $X$ and so called simplex, i.e., a metric space all whose non-zero distances are the same. In the case when the simplex's cardinality does not exceed the cardinality of $X$, a new formula for this distance is obtained. The latter permits to derive an exact formula for the distance between a simplex and an ultrametric space.
\end{abstract}

\section*{Introduction}
\markright{\thesection.~Introduction}
A natural and rather widespread mathematical approach to compare some objects is to define a distance function between them as a measure of their  ``unlikeness'', see many examples in~\cite{DezaDeza}. In the present paper the geometry of the class of all metric spaces considered up to an isometry is investigated by means of the Gromov--Hausdorff distance.

As early as 1914, F.~Hausdorff~\cite{Hausdorff} defined a non-negative symmetric function on pairs of non-empty subsets of a metric space $X$ that is equal to the infimum of non-negative reals $r$ such that one subset is contained in the $r$-neighborhood of the other, and vice versa. Later on D.~Edwards~\cite{Edwards} and M.~Gromov~\cite{Gromov} independently generalized the Hausdorff construction to the family of all compact metric spaces in terms of their isometrical embeddings into all possible ambient spaces, see definition below. The resulting function is called the Gromov--Hausdorff distance, and the corresponding metric space $\cM$ of compact metric spaces considered up to a isometry is referred as the Gromov--Hausdorff space. The geometry of this space turns out to be rather tricky, and it is actively investigated by specialists, because, in particular, the ``space of all spaces'' has several evident applications. It is well-known that $\cM$ is path-connected, complete, separable, geodesic metric space, and that it is not proper. A detailed introduction to geometry of the Gromov--Hausdorff space can be found in~\cite[Ch.~7]{BurBurIva}.

The problem to calculate the Gromov--Hausdorff distance between two given spaces is rather non-trivial. In the present paper the authors continue to study this problem in the particular case where one of the given spaces is so-called simplex, i.e., a metric space all whose non-zero distances are the same. This case turns out to be of special interest due to several reasons. In the case of finite metric space $X$ the distances from $X$ to simplexes permit to reproduce the edges lengths of a minimal spanning tree of $X$, see~\cite{TuzMST-GH}; these distances turn out  to be useful in investigation of isometries of the Gromov--Hausdorff space, see~\cite{ITAutVesnik};  in terms of those distances the generalized Borsuk problem can be solved, see~\cite{IvaTuzBorsuk}.

In paper~\cite{IvaTuzSimpDist} the distances between simplexes and compact metric spaces are calculated in several particular cases. Later on these results are generalized to the case of arbitrary bounded metric spaces, see~\cite{GrigIvaTuz}. Namely, several additional characteristics of the bounded metric spaces were introduced, and in terms of those characteristic either exact formulas for the Gromov--Hausdorff distance to simplexes were written, or exact lower and upper estimates for these distances were given.

In the present paper the formulas from~\cite{GrigIvaTuz} for the Gromov--Hausdorff distance from an arbitrary bounded metric space $X$ to a simplex get a geometrical interpretation, that permits to rewrite them in a more convenient from in the case of simplexes having at most the same cardinality as the space $X$, see Theorem~\ref{thm:extr}. As an application, an exact formulas for the distances from a simplex to an arbitrary finite ultrametric space are obtained, see Theorem~\ref{thm:ultra}. (Recall that a metric space is said to be ultrametric, if the triangle formed by any its three points is an isosceles one, and its ``base'' does not exceed its  ``legs'', see also the definition below).  In addition, a criterion for a finite metric space to be ultrametric in terms of minimal spanning trees is obtained (Theorem~\ref{thm:ultra-and-mst}).

The work is partly supported by President RF Program supporting leading scientific schools of Russia (Project NSh--6399.2018.1, Agreement~075--02--2018--867), by RFBR, Project~19-01-00775-a, and also by MGU scientific schools support program.

\section{Preliminaries}
\markright{\thesection.~Preliminaries}
Let $X$ be an arbitrary set. By $\#X$ we denote the \emph{cardinality\/} of the set $X$.

Let $X$ be an arbitrary metric space. The distance between any its points $x$ and $y$ we denote by $|xy|$. If $A,B\ss X$ are non-empty subsets of $X$, then put $|AB|=\inf\bigl\{|ab|:a\in A,\,b\in B\bigr\}$. For $A=\{a\}$, we write $|aB|=|Ba|$ instead of $|\{a\}B|=|B\{a\}|$.

For each point $x\in X$ and a number $r>0$, by $U_r(x)$ we denote the open ball with center $x$ and radius $r$; for any non-empty $A\ss X$ and a number $r>0$ put $U_r(A)=\cup_{a\in A}U_r(a)$.

\subsection{Hausdorff and Gromov--Hausdorff Distances}
For non-empty $A,\,B\ss X$ put
\begin{multline*}
d_H(A,B)=\inf\bigl\{r>0:A\ss U_r(B),\ \text{and}\ B\ss U_r(A)\bigr\}\\
=\max\{\sup_{a\in A}|aB|,\ \sup_{b\in B}|Ab|\}.
\end{multline*}
This value is called the \emph{Hausdorff distance between $A$ and $B$}. It is well-known, see~\cite{BurBurIva}, that the Hausdorff distance is a metric on the set of all non-empty bounded closed subsets of $X$.

Let $X$ and $Y$ be metric spaces. A triple $(X',Y',Z)$ consisting of a metric space $Z$ together with its subsets $X'$ and $Y'$ isometric to $X$ and $Y$, respectively, is called a \emph{realization of the pair $(X,Y)$}. The \emph{Gromov--Hausdorff distance $d_{GH}(X,Y)$ between $X$ and $Y$} is the infimum of real numbers $r$ such that there exists a realization  $(X',Y',Z)$ of the pair $(X,Y)$ with $d_H(X',Y')\le r$. It is well-known~\cite{BurBurIva} that $d_{GH}$ is a metric on the set $\cM$ of all compact metric spaces considered up to an isometry.

A metric space $X$ is called a \emph{simplex}, if all its non-zero distances are the same. By $\l\D$ we denote a simplex all whose non-zero distances equal $\l>0$. For $\l=1$ the space $\l\D$ is denoted by $\D$ to be short.

\begin{prop}[\cite{BurBurIva}]\label{prop:GH_simple}
Let $X$ be an arbitrary metric space, and $\D$ be a single-point space, then for any $\l>0$ it holds $d_{GH}(\l\D,X)=\frac12\diam X$.
\end{prop}

\begin{thm}[\cite{GrigIvaTuz}]\label{thm:dist-n-simplex-bigger-dim}
Let $X$ be an arbitrary bounded metric space, and $\#X<\#\l\D$, then
$$
2d_{GH}(\l\D,X)=\max\{\l,\diam X-\l\}.
$$
\end{thm}

Let  $X$ be a set, and $m$ a cardinal number that does not exceed $\#X$. By $\cD_m(X)$ we denote the family of all possible partitions of the set $X$ into $m$ subsets.

Now, let $X$ be a metric space. Then for each $D=\{X_i\}_{i\in I}\in\cD_m(X)$ put
$$
\diam D=\sup_{i\in I}\diam X_i.
$$
Further, for any non empty $A,B\ss X$ put
$$
|AB|=\inf\bigl\{|ab|:(a,b)\in A\x B\bigr\},
$$
and for any $D=\{X_i\}_{i\in I}\in\cD_m(X)$ define
$$
\a(D)=\inf\bigl\{|X_iX_j|:i\ne j\bigr\}.
$$

\begin{prop}\label{prop:GH-dist-alpha-beta}
Let $X$ be an arbitrary bounded metric space, and $m=\#\l\D\le\#X$. Then
$$
2d_{GH}(\l\D,X)=\inf_{D\in\cD_m(X)}\max\{\diam D,\,\l-\a(D),\,\diam X-\l\}.
$$
\end{prop}

For an arbitrary metric space $X$ put
$$
\e(X)=\inf\bigl\{|xy|:x,y\in X,\,x\ne y\bigr\}.
$$
Notice that $\e(X)\le\diam X$, and for a bounded metric space $X$ the equality holds if and only if $X$ is a simplex.

\begin{thm}[\cite{IvaTuzSimpDist}]\label{thm:dist-n-simplex-same-dim}
Let $X$ be a finite metric space, and $\#\l\D=\#X$, then
$$
2d_{GH}(\l\D,X)=\max\bigl\{\l-\e(X),\,\diam X-\l\bigr\}.
$$
\end{thm}

For an arbitrary metric space $X$, $m=\#X$, put
\begin{gather*}
\a_m^-(X)=\inf_{D\in\cD_m(X)}\a(D),\qquad \a_m(X)=\a_m^+(X)=\sup_{D\in\cD_m(X)}\a(D),\\
d_m(X)=d_m^-(X)=\inf_{D\in\cD_m(X)}\diam D,\qquad d_m^+(X)=\sup_{D\in\cD_m(X)}\diam D.
\end{gather*}

\begin{rk}
We use short notations for $\a_m^+(X)$ and $d_m^-(X)$, because in the formulas given below these two values appears more often than there ``twins'' $\a_m^-(X)$ and $d_m^+(X)$.
\end{rk}

\begin{thm}[\cite{GrigIvaTuz}]\label{thm:collective}
Let $X$ be an arbitrary bounded metric space, and $m=\#\l\D\le\#X$.
\begin{enumerate}
\item\label{thm:collective:1} If $2d_m^+(X)>\diam X-\a_m(X)$, then put
\begin{align*}
&a=\max\biggl\{\a_m^-(X)+d_m(X),\,\frac{\diam X+\a_m^-(X)}2,\,\diam X-d_m^+(X)\biggr\},\\
&b=\a_m(X)+d_m^+(X).
\end{align*}
Then $a<b$, and
\begin{enumerate}
\item\label{thm:collective:1:1} for $\l\le a$ the equality
$$
2d_{GH}(\l\D,X)=\max\bigl\{\diam X-\l,\,d_m(X)\bigr\}
$$
holds\rom;
\item\label{thm:collective:1:2} for $a\le\l\le b$ the exact estimations
\begin{multline*}
\max\bigl\{\diam X-\l,\,d_m(X),\,\l-\a_m(X)\bigr\}\le2d_{GH}(\l\D,X)\\
\le\min\bigl\{d_m^+(X),\,\l-\a_m^-(X)\bigr\}
\end{multline*}
hold\rom;
\item\label{thm:collective:1:3} for $\l\ge b$ the equality $2d_{GH}(\l\D,X)=\l-\a_m(X)$ holds.
\end{enumerate}
\item\label{thm:collective:2} If $2d_m^+(X)\le\diam X-\a_m(X)$, then
$$
2d_{GH}(\l\D,X)=\max\bigl\{\diam X-\l,\,\l-\a_m(X)\bigr\}.
$$
\end{enumerate}
\end{thm}

In computing the Gromov--Hausdorff distances between finite metric spaces, minimal spanning trees turn out to be rather important. The reason is that the edges lengths of these trees are related to geometrical characteristics of partitions of the corresponding ambient space. Recall the definitions.

Let  $G=(V,E)$ be an arbitrary (simple) graph with a vertex set $V$ and an edge set $E$. If $V$ is a metric space, then the \emph{lengthes $|e|$ of edges $e=vw$ of $G$} are defined as the distances $|vw|$ between their ends $v$ and $w$ in the space $X$; the \emph{length  $|G|$ of the graph $G$} is defined as the sum of the lengths of all its edges.

Let $X$ be a finite metric space. Define the number $\mst(X)$ as the length of a shortest tree of the form $(X,E)$. The resulting value is referred as the  \emph{length of a minimal spanning tree on $X$}. Each tree $G=(X,E)$ such that  $|G|=\mst(X)$ is called a  \emph{minimal spanning tree on $X$}. Notice that for any such $X$ a minimal spanning tree on $X$ does always exist. By  $\MST(X)$ we denote the set of all minimal spanning trees on $X$.

Notice that, generally speaking, a minimal spanning tree is not unique.  For $G\in\MST(X)$ by  $\s(G)$ we denote the vector composed of the edges lengths of the tree $G$, that are written in the decreasing order.

\begin{prop}\label{prop:mst-spect}
For any $G_1,G_2\in\MST(X)$, the equality $\s(G_1)=\s(G_2)$ holds.
\end{prop}

Proposition~\ref{prop:mst-spect} implies that following definition is correct.

\begin{dfn}
For any finite metric space $X$, by $\s(X)$ we denote the vector $\s(G)$ for an arbitrary  $G\in\MST(X)$, and we call it the \emph{$\mst$-spectrum of the space $X$}.
\end{dfn}

\section{The Gromov--Hausdorff Distance to Simplexes in Terms of Extreme Points}
\markright{\thesection.~The Gromov--Hausdorff Distance in Terms of Extreme Points}

Let $X$ be a bounded metric space $X$, and $m\le\#X$ a cardinal number.  In accordance with Proposition~\ref{prop:GH-dist-alpha-beta}, to calculate the exact value of the function $g_m(\l)=d_{GH}(\l\D,X)$ it suffices to know the pairs $\bigl(\a(D),\diam D\bigr)$ for all $D\in\cD_m(X)$. Let us consider these pairs as points in the plane with fixed standard coordinates $(\a,d)$. By $\cAD_m(X)\ss\R^2$ we denote the set of all such pairs, and also put $h_{\a,d}(\l):=\max\{d,\,\l-\a\}$. Reformulate Proposition~\ref{prop:GH-dist-alpha-beta} in the new notations using the fact that the function $\diam X-\l$ does not depend on partitions $D$, and changing the order of $\inf$ and $\max$.

\begin{cor}\label{cor:GH-dist-alpha-diam-set}
Let $X$ be an arbitrary bounded metric space, and $m=\#\l\D\le\#X$. Then
$$
2d_{GH}(\l\D,X)=\max\bigl\{\diam X-\l,\,\inf_{(\a,d)\in\cAD_m(X)}h_{\a,d}(\l)\bigr\}.
$$
\end{cor}

In what follows it is convenient to work with the closure $\bcAD_m(X)$ of the set $\cAD_m(X)$. Notice that as $\cAD_m(X)$, so as $\bcAD_m(X)$ lie in the parallelogram formed by the intersection of two strips: the horizontal strip between the straight lines $y=d_m^-(X)$ and $y=d_m^+(X)$, the slant strip between the straight lines $y=\l-\a_m^-(X)$ and $y=\l-\a_m^+(X)$. Thus, the both sets are bounded, and the set $\bcAD_m(X)$ is compact.

Put
\begin{equation}\label{eq:FbF}
F(\l)=\inf_{(\a,d)\in\cAD_m(X)}h_{\a,d}(\l),\ \
\bF(\l)=\inf_{(\a,d)\in\bcAD_m(X)}h_{\a,d}(\l).
\end{equation}
Since for each fixed  $\l$ the value $h_{\a,d}(\l)$ depends continuously on $\a$ and $d$, then $F(\l)=\bF(\l)$, and hence the following result holds.

\begin{cor}\label{cor:GH-dist-alpha-diam-set-closure}
Let $X$ be an arbitrary bounded metric space and $m=\#\l\D\le\#X$. Then
$$
2d_{GH}(\l\D,X)=\max\bigl\{\diam X-\l,\,\inf_{(\a,d)\in\bcAD_m(X)}h_{\a,d}(\l)\bigr\}.
$$
\end{cor}

Further, notice that for any point $(\a,d)\in\R^2$, the graph of the function $y=h_{\a,d}(\l)$ is an angle with the vertex at the point $T_{\a,d}=(\a+d,d)$; one side of the angle is horizontal, and its direction is opposite to the one of the abscissa axis; the second side has the same direction as the bisectrix of the first quadrant. Notice that for any $(\a,d),\,(\a',d')$ such that $\a\ge\a'$ and $d\le d'$ the inequality $h_{\a,d}(\l)\le h_{\a',d'}(\l)$ holds for all $\l$, thus, calculating $d_{GH}(\l\D,X)$, we can omit $(\a',d')\in\bcAD_m(X)$ such that there exists a pair $(\a,d)\in\bcAD_m(X)$ with $\a\ge\a'$, $d\le d'$.

A point $(\a,d)\in\bcAD_m(X)$ is said to be \emph{extreme}, if there is no another point $(\a',d')\in\bcAD_m(X)$, $(\a',d')\ne (\a,d)$, such that $\a'\ge\a$ and $d'\le d$. By $\Ext_m(X)$ we denote the set of all extreme points from $\bcAD_m(X)$.

\begin{rk}\label{rk:oder}
Define an ordering relation in the plane as follows: $(\a,d)\le(\a',d')$ iff $\a\ge\a'$ and $d\le d'$. Then for each fixed $\l$ the mapping $(\a,d)\mapsto h_{\a,d}(\l)$ is monotonic. A point $(\a,d)\in\bcAD_m(X)$ is extreme, if and only if it is a maximal element of the set $\bcAD_m(X)$ with respect to this ordering.
\end{rk}

\begin{prop}\label{prop:Extr}
The set $\Ext_m(X)$ is not empty.
\end{prop}

\begin{proof}
As it is already mentioned above, the set $\bcAD_m(X)$ is compact, and hence the continuous function $\pi_2\:\bcAD_m(X)\to\R$, $\pi_2\:(\a,d)\mapsto d$, takes its least value (which is equal to $d_m(X)$) at it. Therefore, the set $M=\bigl\{(\a,d)\in\bcAD_m(X):\pi_2(\a,d)=d_m(X)\bigr\}$ is not empty and, due to continuity of the function $\pi_2$, it is compact, so the function $\pi_1\:M\to\R$, $\pi_1\:(\a,d)\mapsto\a$, takes its greatest value at some point $\bigl(\a_0,d_m(X)\bigr)\in\bcAD_m(X)$. It is clear that this point is extreme.
\end{proof}

\begin{thm}\label{thm:extr}
Let $X$ be an arbitrary bounded metric space, and $m=\#\l\D\le\#X$. Then
$$
2d_{GH}(\l\D,X)=\max\bigl\{\diam X-\l,\,\inf_{(\a,d)\in\Ext_m(X)}h_{\a,d}(\l)\bigr\}.
$$
\end{thm}

\begin{proof}
Let $\bF(\l)$ be given by the formula form~(\ref{eq:FbF}). Put
$$
H(\l)=\inf_{(\a,d)\in\Ext_m(X)}h_{\a,d}(\l).
$$

Since $\Ext_m(X)\ss\bcAD_m(X)$, then for all $\l$ we have $H(\l)\ge\bF(\l)$. Now, prove the inverse inequality.

Fix an arbitrary $\l$. By definition of the function $\bF(\l)$, there exists a sequence $(\a_i,d_i)\in\bcAD_m(X)$ such that $h_{\a_i,d_i}(\l)\to\bF(\l)$. Since the set  $\bcAD_m(X)$ is compact, then we can assume (passing to a subsequence if it is necessary) that the sequence $(\a_i,d_i)$ tends to some $(\a',d')\in\bcAD_m(X)$. Since $h_{\a,d}(\l)$ is continuous with respect to $(\a,d)$, them we have $h_{\a',d'}(\l)=\bF(\l)$.

Put
$$
\bcAD_m(X,\l)=\bigl\{(\a',d')\in\bcAD_m(X):h_{\a',d'}(\l)=\bF(\l)\bigr\}.
$$
As it is already shown, the set $\bcAD_m(X,\l)$ is non-empty. Let us prove that this set contains an extreme point.

The continuity of the function $h_{\a,d}(\l)$ with respect to $(\a,d)$ implies that $\bcAD_m(X,\l)$ is closed, and hence it is compact. Put $d_0=\inf\bigl\{d:(\a,d)\in\bcAD_m(X,\l)\bigr\}$. Since the set $\bcAD_m(X,\l)$ is compact, then there exists an $\a$ such that $(\a,d_0)\in\bcAD_m(X,\l)$.

By $\bcAD_m(X,\l,d_0)$ we denote the set of all such pairs $(\a,d_0)$. As it is shown above, this set is non-empty. The continuity of the function $h_{\a,d}(\l)$ with respect to $\a$ implies that $\bcAD_m(X,\l,d_0)$ is a compact set. Put $\a_0=\sup\{\a:(\a,d_0)\in\bcAD_m(X,\l,d_0)\}$. Since the set  $\bcAD_m(X,\l,d_0)$ is compact, then the point $(\a_0,d_0)$ belongs to $\bcAD_m(X,\l,d_0)\ss\bcAD_m(X)$. However, this point is extreme, because otherwise we obtain a contradiction with the definitions of $d_0$ and $\a_0$. The latter implies that $H(\l)\le\bF(\l)$, the proof is complete.
\end{proof}

\section{The Gromov--Hausdorff Distance between Simplexes and Ultrametric Spaces}
\markright{\thesection.~The Case of Ultrametric Spaces}
Now apply the above results to obtain specific explicit formulas for the Gromov--Hausdorff distance in the particular case of  \emph{ultrametric\/} spaces, i.e\., the metric spaces $X$ such that the distance function satisfies the following enforced triangle inequality:
$$
|xz|\le\max\bigl\{|xy|,\,|yz|\bigr\}
$$
for all $x,y,z\in X$. This inequality implies similar ``polygon inequality'', namely, for an arbitrary set of points $x_1, x_2,\ldots,x_k$ of the space $X$, the inequality
$$
|x_1x_k|\le\max\bigl\{|x_1x_2|,|x_2x_3|,\ldots,|x_{k-1}x_k|\bigr\}
$$
holds.

\begin{thm}\label{thm:ultra-and-mst}
Let $X$ be a finite metric space, and $G$ a minimal spanning tree on $X$. For any distinct $v,w\in X$, by $\g_{vw}$ we denote the unique path in $G$ connecting $v$ and $w$. Then $X$ is ultrametric if and only if for any distinct points $v,w\in X$ the equality
\begin{equation}\label{eq:ultra-crit}
|vw|=\max_{e\in E(\g_{vw})}|e|
\end{equation}
holds.
\end{thm}

\begin{proof}
At first, let $X$ be an ultrametric space. Choose arbitrary distinct $v,w\in X$. If $vw\in E(G)$, then Equality~(\ref{eq:ultra-crit}) holds. Now, let $v$ and $w$ be not adjacent. Since $G$ is a minimal spanning tree, then $|vw|\ge\max_{e\in E(\g_{vw})}|e|$. The inverse inequality is the  ``polygon inequality'' in an ultrametric space.

Conversely, let Equality~(\ref{eq:ultra-crit}) hold for any distinct $v,w\in E(G)$. Choose any three pairwise distinct points $u,v,w\in X$. If one of them, say $v$, lies in the path $G$ connecting two other points, i.e\., in the path $\g_{u,w}$, then
$$
|uw|=\max_{e\in E(\g_{uw})}|e|=\max_{e\in E(\g_{uv})\cup E(\g_{vw})}|e|=\max\bigl\{|uv|,|vw|\bigr\},
$$
and hence the ``triangle'' $uvw$ is isosceles, and its ``base'' is not longer than its ``legs''.

Now, suppose that no one of the points under consideration lies in the path in $G$ connecting two other points. Then the three paths in $G$ connecting pairwise the vertices $u,v,w$ have a common point $x\in X$ distinct form these three. Therefore, each of those three paths can be represented as a union of the pair of paths connecting $x$ and $u,v,w$. We have
\begin{multline*}
|uw|=\max_{e\in E(\g_{uw})}|e|=\max_{e\in E(\g_{ux})\cup E(\g_{xw})}|e|\\
\le\max_{e\in E(\g_{ux})\cup E(\g_{xw})\cup E(\g_{xv})}|e|=\max\bigl\{|uv|,|vw|\bigr\},
\end{multline*}
the proof is complete.
\end{proof}

Theorem~\ref{thm:ultra-and-mst} immediately implies the following result.

\begin{cor}\label{cor:ultra-and-mst-sp}
Let $X$ be a finite ultrametric space consisting of $n$ points, and $\s(X)=(\s_1,\s_2,\ldots,\s_{n-1})$ its $\mst$-spectrum. Then $|vw|\in\s(X)$ for any distinct $v,w\in X$. In particular, $\diam X=\s_1$.
\end{cor}

\begin{thm}\label{thm:ultra}
Let $X$ be a finite ultrametric space consisting of $n$ points, $\s(X)=(\s_1,\s_2,\ldots,\s_{n-1})$ its $\mst$-spectrum, and $m>0$ a cardinal number. Then
$$
2d_{GH}(\l\D_m,X)=
\begin{cases}
\s_1&\text{for $m=1$},\\
\max\{\s_1-\l,\,\s_m,\,\l-\s_{m-1}\}&\text{for $1<m<n$},\\
\max\{\s_1-\l,\,\l-\s_{n-1}\}&\text{for $m=n$},\\
\max\{\s_1-\l,\,\l\}&\text{for $m>n$}.
\end{cases}
$$
\end{thm}

\begin{proof}
Taking into account Corollary~\ref{cor:ultra-and-mst-sp}, the first formula follows from Proposition~\ref{prop:GH_simple}, the third and the forth ones follow from Theorem~\ref{thm:dist-n-simplex-same-dim} and Theorem~\ref{thm:dist-n-simplex-bigger-dim}, respectively. Let us prove the second formula.

To do this we apply Theorem~\ref{thm:extr} and show that the set $\Ext_m(X)$ consists of a single point, namely of the point $(\s_{m-1},\s_m)$. In other words, the point $(\s_{m-1},\s_m)$ is the least element of the set $\bcAD_m(X)$ with respect to the ordering defined in Remark~\ref{rk:oder}.

Let $G$ be an arbitrary minimal spanning tree on $X$, and let $e_1,\ldots,e_{n-1}$ be its edges, where $|e_i|=\s_i$.

Through out the edges $e_1,\ldots,e_{m-1}$ from $G$, then $G$ splits into $m$ trees $G_1,\ldots,G_m$. Put  $X_i=V(G_i)$, then $D=\{X_i\}_{i=1}^m\in\cD_m(X)$. In accordance with Theorem~\ref{thm:ultra-and-mst}, the diameters of all $X_i$ do not exceed $|e_m|$, and one of  $G_i$ contains $e_m$, hence the diameter of the corresponding $X_i$ does equal $|e_m|$. Thus, $\diam D=|e_m|=\s_m$.

On the other hand, for any $x_i\in X_i$ and $x_j\in X_j$, $i\ne j$, the path $\g$ in $G$ connecting $x_i$ and $x_j$ passes through some edges from the set $\{e_1,\ldots,e_{m-1}\}$, therefore, in accordance with Theorem~\ref{thm:ultra-and-mst}, $|x_ix_j|\ge|e_{m-1}|$. Besides, there exist some $X_i$ and $X_j$ that are connected by the edge $e_{m-1}$, therefore $|X_iX_j|=|e_{m-1}|$, and so, $\a(D)=|e_{m-1}|=\s_{m-1}$.

Now, let us show that the point $(\s_{m-1},\s_m)$ is extreme, and that any other point  $\bigl(\a(D'),\diam D'\bigr)$, $D'\in\cD_m(X)$, is not extreme.

Choose an arbitrary $D'=\{Y_i\}_{i=1}^m\in\cD_m(X)$, and construct the graph $G'$ whose vertices are the sets $Y_i$, and whose edges are the edges from $E(G)$ that connect distinct sets $Y_i$ and $Y_j$. Then $G'$ is a connected graph (probably containing multiple edges), therefore it has at least $(m-1)$ edges $e_i$. Therefore, the shortest edge of the graph $G'$ can not be longer than $|e_{m-1}|$, and so $\a(D')\le|e_{m-1}|$.

It remains to show that $\diam D'\ge|e_m|$. Assume the contrary, then no one of the edges $e_i$, $i\le m$, connects vertices from the same $Y_j$. The latter implies that all such $e_i$ connect distinct $Y_j$ and $Y_k$, i.e., they are edges of the graph $G'$. Therefore, the connected graph $G'$ contains some cycle $C$. Let $Y_{i_1},\ldots,Y_{i_k},Y_{i_{k+1}}=Y_{i_1}$ be consecutive vertices of the cycle $C$, and let $e_{i_p}$ be the edge of $C$, connecting $Y_{i_p}$ and $Y_{i_{p+1}}$. Put $e_{i_p}=v_pw_p$, where $v_p\in Y_{i_p}$, $w_p\in Y_{i_{p+1}}$.

Let $\g_p$, $p=1,\ldots,k-1$, be the path in the tree $G$, connecting $w_p$ and $v_{p+1}$. Since the both vertices belong to $Y_{i_{p+1}}$ and $\diam Y_{i_{p+1}}<|e_m|$ in accordance with assumptions, then, due to Theorem~\ref{thm:ultra-and-mst}, the path $\g_p$ does not pass through any edge from the set $\{e_i\}_{i=1}^m$. Therefore, passing consequently the edges and paths $e_{i_1},\g_1,e_{i_2},\g_2,\ldots,\g_{k-1},e_{i_k}$, we get a path in $G$, that connects the vertices $v_1$ and $w_k$ from $Y_{i_1}$. Hence, due to Theorem~\ref{thm:ultra-and-mst}, we have $\diam Y_{i_1}\ge|v_1w_k|\ge|e_m|$, a contradiction. Theorem is proved.
\end{proof}

\end{document}